\documentclass[12pt,a4paper]{amsart}
\usepackage{mathrsfs}
 \usepackage{amsfonts}
 \usepackage{amsthm}
 \usepackage[dvips]{graphicx}
 \usepackage{subfigure}
 \usepackage[all]{xy}
 \usepackage{float}
 \usepackage{cases}
\usepackage{psfrag}
\usepackage[colorlinks=true,pdfpagemode=FullScreen]{hyperref}
\usepackage{amscd}
\usepackage{amssymb}
\usepackage{verbatim}
\usepackage{amsmath}
\newtheorem{thm}{Theorem}[section]
\newtheorem{cor}[thm]{Corollary}
\newtheorem{lem}[thm]{Lemma}
\newtheorem{prop}[thm]{Proposition}
\newtheorem{defn}[thm]{Definition}
\newtheorem{rem}[thm]{Remark}

\newcommand{\Z}{\mathbb Z}
\newcommand{\Q}{{\mathbb Q}}
\newcommand{\R}{{\mathbb R}}
\newcommand{\conv}{\mathop{\mathrm{co}}}
\newcommand{\dom}{\mathop{\mathrm{dom}}}
\newcommand{\im}{\mathop{\mathrm{Im}}}

\def\bc{\begin{center}}       \def\ec{\end{center}}
\def\be{\begin{equation}}     \def\ee{\end{equation}}
\def\ba{\begin{array}}        \def\ea{\end{array}}
\def\bea{\begin{eqnarray}}            \def\eea{\end{eqnarray}}
\def\beaa{\begin{eqnarray*}}  \def\eeaa{\end{eqnarray*}}

\def\p{\partial}
\def\sp{f^{\infty}}
\def\sn{f^{-\infty}}
\def\m{\underline{\gamma}}
\def\s{\overline{\gamma}}
\def\T{\Xi}

\begin{document}
\title{Subtleties of the minmax selector}
\author{WEI Qiaoling}

\begin{address}{WEI Qiaoling, Institut de Math\'{e}matiques de
Jussieu, Universit\'{e} Paris 7, 175 rue du chavaleret, 75013,
Paris, France} \email{weiqiaoling@math.jussieu.fr}
\end{address}

\date{\today}
 \maketitle
\begin{abstract}
In this note, we show that the minmax and maxmin critical values of
a function quadratic nondegenerate at infinity are equal when
defined in homology or cohomology with coefficients in a field.
However, by an example of F.~Laudenbach, this is not always true for
coefficients  in a ring and, even in the case of a field, the
minmax-maxmin depends on the  field.
\end{abstract}
\bigskip
\section{Introduction}
Given a Lagrangian submanifold $L$ in the cotangent bundle of a
closed manifold $M$, obtained by Hamiltonian deformation of the zero
section, the minmax selector introduced by J.-C.~Sikorav provides an
almost everywhere defined section $M\rightarrow L$ of the projection
$T^*M\rightarrow M$ restricted to $L$. As noticed by M.~Chaperon
\cite{MC}, this defines weak solutions of smooth Cauchy problems for
Hamilton-Jacobi equations; in the classical case of a convex
Hamiltonian, the minmax is a minimum and the minmax solution
coincides with the viscosity solution, which is not always the case
for nonconvex Hamiltonians. For a recent use of the minmax selector
in weak KAM theory, see \cite{arnaud}.

The minmax has been defined using homology or cohomology with
various coefficient rings, for example $\Z$ in \cite{MC,Vi},  $\Q$
in \cite{Ca} and  $\Z_2$ in \cite{Pl}. Also, in \cite{Vi}, the
maxmin was mentioned as a natural analogue to the minmax.  But there
is no evidence showing that all these critical values coincide.
G.~Capitanio has given a proof  \cite{Ca} that the maxmin and minmax
for homology with coefficients in $\Q$ are equal, but the criterion
he uses (Proposition $2$ in \cite{Ca}) is not correct---see
Remark~\ref{rem42} hereafter.

In this note, we investigate the maxmin and minmax for a general
function quadratic at infinity, not necessarily related to
Hamilton-Jacobi equations. We give both algebraic and geometric
proofs that the minmax and maxmin with coefficients in a field
coincide; the geometric proof, based on Barannikov's Jordan normal
form for the boundary operator of the Morse complex, improves our
understanding of the problem.

A counterexample for coefficients in $\Z$\/, due to F.~Laudenbach,
is constructed using Morse homology; in this example, moreover, the
minmax-maxmin for  coefficients in $\Z_2$ is not the same as for
coefficients in $\Q$\/. However, if the minmax and maxmin for
coefficients in $\Z$ coincide, then all three minmax-maxmin critical
values are equal.

\section{Maxmin and Minmax}

\subsection*{Hypotheses and notation}
We denote by $X$ the vector space $\R^n$ and by $f$ a real function
on $X$\/, \emph{quadratic at infinity} in the sense that it is
continuous and there exists a nondegenerate quadratic form $Q:X
\to\R$ such that $f$ coincides with $Q$ outside a compact subset.

Let $f^c:=\{x| f(x)\leq c\}$ denote the sub-level sets of $f$\/.
Note that for $c$ large enough, the homotopy types of $f^c$,
$f^{-c}$ do not depend on $c$, we may denote them as $f^{\infty}$
and $f^{-\infty}$. Suppose the quadratic form $Q$ has Morse index
$\lambda$, then the homology groups with coefficient ring $R$ are
\[H_*(\sp,\sn; R)
\simeq \left\{
              \begin{array}{ll}
               R & \textrm{in dimension}\, \lambda \\
               0 &\textrm{otherwise}
              \end{array}
                        \right.
\]
Consider the  homomorphism of homology groups
\[i_{c*}: H_*(f^c,\sn;R)\to H_*(\sp,\sn;R)\]
induced by the inclusion $i_c: (f^c,\sn)\hookrightarrow
(\sp,\sn)$\/.

\begin{defn}\rm  If $\Xi$ is a generator of
$H_{\lambda}(\sp,\sn;R)$, we let
\[\m(f,R):=\inf\{c: \Xi\in \im (i_{c*})\}\/,\]
i.e. $\m(f,R)=\inf\{c: i_{c*}H_\lambda(f^c,\sn;R)=
H_{\lambda}(\sp,\sn;R)\}$\/.
\end{defn}

Similarly, we can consider the homology group
\[H_*(X\setminus\sn,X\setminus\sp;R)\simeq \left\{
              \begin{array}{ll}
               R, & \textrm{in dimension}\, n-\lambda \\
               0, &\textrm{otherwise}
              \end{array}
                        \right.\]
and the homomorphism
\[j_{c*}: H_*(X\setminus f^c,X\setminus \sp;R)\to
H_*(X\setminus\sn,X\setminus\sp;R)\] induced by $j_c:(X\setminus
f^c,X\setminus \sp)\hookrightarrow (X\setminus \sn,X\setminus \sp)$.

\begin{defn}\rm  If $\Delta$ is a generator of
$H_{n-\lambda}(X\setminus\sn,X\setminus\sp;R)$\/, we let
\begin{eqnarray*}
\s(f,R) & := & \sup\{c: \Delta\in \im (j_{c*})\} \\
 & = & \sup\{c: j_{c*}H_{n-\lambda}(X\setminus f^c,X\setminus\sp;R)=H_{n-\lambda}(X\setminus\sn,X\setminus\sp;R)\}.
\end{eqnarray*}
\end{defn}

\begin{lem}
\label{max} One has that
 \beaa \m(f,R)&=&\inf\max f:=\inf_{[\sigma]=
\Xi}\max_{x\in |\sigma|}f(x)\\
\s(f,R)&=&\sup\min f:=\sup_{[\sigma]=\Delta}\min_{x\in
|\sigma|}f(x),\eeaa where $\sigma$ is a relative cycle and
$|\sigma|$ denotes its support. We call $\sigma$ a descending (resp.
ascending) simplex if $[\sigma]=\Xi$ (resp. $[\sigma]=\Delta$).
\end{lem}

\begin{proof} A descending simplex $\sigma$ defines a homology class in $H_{\lambda}(f^c,f^{-\infty};R)$ if and only if $|\sigma|\subset f^c$\/, in which case one has $\displaystyle\max_{x\in |\sigma|}f(x)\leq c$\/, hence $\m(f,R)\geq\inf\max f$\/;  choosing $c=\displaystyle\max_{x\in |\sigma|}f(x)$\/, we get equality. The case of $\s$ is identical.
\end{proof}

\begin{defn}\rm  $\m(f,R)$ is called a minmax of $f$ and $\s(f,R)$\/, a
maxmin .
\end{defn}
\begin{rem}
As we shall see later, in view of Morse homology, the names are
proper generically for Morse-excellent functions.
\end{rem}
One can also consider cohomology instead of homology and define
\beaa
\underline{\alpha}(f,R)&:=&\inf\{c: i_c^*\neq 0\}\/,\quad i_c^*:H^\lambda(f^\infty,f^{-\infty};R)\rightarrow H^\lambda(f^c,f^{-\infty};R)
\\
\overline{\alpha}(f,R)&:=&\sup\{c: j_c^*\neq 0\}\/,\quad
j_c^*:H^{n-\lambda}(X\setminus f^{-\infty},X\setminus
f^\infty;R)\rightarrow H^{n-\lambda}(X\setminus f^c,X\setminus
f^\infty;R).
\eeaa

\begin{prop}[\cite{Vi}, Proposition 2.4]
\label{Vi} When $X$ is $R$-oriented,
\[
\overline{\alpha}(f,R)=\m(f,R)\quad\hbox{and}\quad\underline{\alpha}(f,R)=\s
(f,R).
\]
\end{prop}
\begin{proof}
We establish for example the first identity: one has the commutative
diagram
\[
\begin{array}{ccc}
H_\lambda(f^c,f^{-\infty};R) & \simeq & H^{n-\lambda}(X\setminus f^{-\infty},X\setminus f^c;R) \\
\downarrow^{i_{c*}} &  & \downarrow \\
H_\lambda(f^\infty,f^{-\infty};R) & \simeq & H^{n-\lambda}(X\setminus f^{-\infty},X\setminus f^\infty;R) \\
\downarrow &  & \downarrow^{j_c^*} \\
H_\lambda(f^\infty,f^c;R) & \simeq & H^{n-\lambda}(X\setminus
f^c,X\setminus f^\infty;R)
\end{array}
\]
where the horizontal isomorphisms are given by Alexander duality
(\cite{H}, section~3.3) and the columns are exact. It does follow
that $i_{c*}$ is onto if and only if $j_c^*$ is zero.
\end{proof}

\begin{defn}
\rm As long as $X$ is finite dimensional, the \emph{Clarke
generalized derivative} of a locally Lipschitzian function $f:X\to
\R$ can be defined as follows:
\[
\p f(x):= \conv\{\lim_{x'\to x} df(x'), \; x'\in \dom(df)\}\/;
\]
where $\conv$ denotes the convex envelop. A point $x\in X$ is called
a \emph{critical point} of $f$ if $0\in \p f(x)$.
\end{defn}

\begin{prop}
If $f$ is $C^2$ then $\m(f,R)$ and $\s(f,R)$ are critical values of
$f$\/; they are critical values of $f$ in the sense of Clarke when
$f$ is locally Lipschitzian.
\end{prop}

\begin{proof}
Take $\m$ for example: if $c=\m(f,R)$ is not a critical value then,
for
 small $\epsilon>0$,  $f^{c-\epsilon}$ is a deformation retract of $f^{c+\epsilon}$ via the flow of $-\frac{\nabla f}{\|\nabla f\|^2}$, hence $\m(f,R)\leq c-\epsilon$, a contradiction. The same argument applies when $f$ is only locally Lipschitzian, replacing $\nabla f$ by a pseudo-gradient.
\end{proof}

\begin{lem}
\label{lem29} If $f$ is locally Lipschitzian, then
\[
\s(f,R)=-\m(-f,R)
\]
\end{lem}

\begin{proof}
Using a (pseudo-)gradient of $f$ as previously, one can see that
$X\setminus f^c$ and $(-f)^{-c}$ have the same homotopy type when
$c$ is not a critical value of $f$. Otherwise, choose a sequence of
non-critical values $c_n\nearrow c=\s(f,R)$, then $-c_n\geq
\m(-f,R)$, taking the limit, we have $\s(f,R)\leq -\m(-f,R)$.
Similarly, taking $c_n'\searrow \m(-f,R)$, then $-c_n'\leq \s(f,R)$,
from which the limit gives us the inverse inequality $-\m(-f,R)\leq
\s(f,R)$.
\end{proof}

The following two questions  arise naturally: \\
(1) Do we have $\m(f,R)=\s(f,R)$?\\
(2) Do $\m(f,R)$ and $\s(f,R)$ depend on the coefficient ring $R$?

Here are two obvious elements for an answer:

\begin{prop}
\label{prop210} One has $\m(f,\Z)\geq\s(f,\Z)$\/.
\end{prop}

\begin{proof}
As the intersection number of $\Xi$ and $\Delta$ is $\pm1$\/, the
support of any descending simplex $\sigma$ must intersect the
support of any ascending simplex $\tau$ at some point $\bar x$\/,
hence $\displaystyle\max_{x\in |\sigma|}f(x)\geq f(\bar
x)\geq\min_{x\in |\tau|}f(x)$\/.
\end{proof}

\begin{prop}
\label{prop211} One has $\m(f,\Z)\geq\m(f,R)$ and
$\s(f,\Z)\leq\s(f,R)$ for every ring $R$\/.
\end{prop}
\begin{proof}
A simplex $\sigma$ whose homology class generates
$H_\lambda(\sp,\sn;\Z)$ induces a simplex whose homology class
generates $H_\lambda(\sp,\sn;R)$\/, hence the first inequality and,
mutatis mutandis, the second one.
\end{proof}

\begin{thm}
\label{sm} If $\mathbb{F}$ is a field, then
$\m(f,\mathbb{F})=\s(f,\mathbb{F})$.
\end{thm}

\begin{proof} By Proposition~\ref{Vi}, it is enough to prove that
\[
\m(f,\mathbb{F})=\underline{\alpha}(f,\mathbb{F})
\]
Recall that $\m(f,\mathbb{F})$ (resp.
$\underline{\alpha}(f,\mathbb{F})$\/) is the infimum of the real
numbers $c$ such that $i_{c*}: H_\lambda (f^c,\sn;\mathbb{F})\to
H_\lambda (\sp,\sn;\mathbb{F})$ is onto (resp. such that $i_c^*:
H^\lambda (\sp,\sn;\mathbb{F})\to H^\lambda (f^c,\sn;\mathbb{F})$ is
nonzero). Now, as $H_\lambda (\sp,\sn;\mathbb{F})$ is a
one-dimensional vector space over $\mathbb{F}$\/, the linear map
$i_{c*}$ is onto if and only if it is nonzero, i.e. if and only if
the transposed map $i_c^*$ is nonzero.\end{proof}

\subsection*{Remark}
This proof is invalid for coefficients in $\Z$ since a $\Z$-linear
map to $\Z$\/, for example $\Z\ni m\to km$\/, $k\in\Z$\/, $k>1$\/,
can be nonzero without being onto; we shall see in
Section~\ref{sec4} that Theorem~\ref{sm} itself is not true in that
case.

\begin{cor}
\label{cor213}
 If $\m(f,\Z)=\s(f,\Z)=\gamma $ then $\m(f,\mathbb{F})=\s(f,\mathbb{F})=\gamma $ for every field $\mathbb{F}$\/.
\end{cor}

\begin{proof}
This follows at once from Theorem~\ref{sm} and
Proposition~\ref{prop211}.
\end{proof}

\begin{cor} \label{da}Let $\gamma\in \R$ have the following property: there exist both a descending simplex over $\Z$ along
which $\gamma$ is the maximum of $f$ and an ascending simplex over
$\Z$ along which $\gamma$ is the minimum of $f$\/. Then,
$\m(f,\Z)=\s(f,\Z)=\m(f,\mathbb{F})=\s(f,\mathbb{F})=\gamma $ for
every field $\mathbb{F}$\/.
\end{cor}

\begin{proof}
We have $\m(f;\Z)\leq \gamma\leq \s(f;\Z)$ by Lemma~\ref{max} and
$\s(f;\Z)\leq \m(f;\Z)$ by Proposition~\ref{prop210}, hence our
result by Corollary~\ref{cor213}.
\end{proof}

\section{Morse complexes and the Barannikov normal form}

The previous proof of Theorem~\ref{sm}, though simple, is quite
algebraic. We now give a more geometric proof, which we find more
concrete and illuminating, based on Barannikov's  canonical form of
Morse complexes. It will provide a good setting for the
counterexample in Section~\ref{sec4}.

First, there is a continuity result for the minmax and maxmin:

\begin{prop} \label{st}
If $f$ and $g$ are two continuous functions quadratic at infinity
with the same reference quadratic form, then \beaa
|\m(f,R)-\m(g,R)|&\leq& |f-g|_{C^0}\\
|\s(f,R)-\s(f,R)|&\leq& |f-g|_{C^0}. \eeaa
\end{prop}
\begin{proof}
For $f\leq g$, from Lemma \ref{max}, it is easy to see that
$\m(f)\leq \m(g)$. In the general case, this implies $\m(g)\leq
\m(f+|g-f|)\leq \m(f)+ |g-f|_{C^0}$\/; exchanging $f$ and $g$\/, we
get $\m(f)\leq\m(g)+|f-g|_{C^0}$.
\end{proof}


\begin{cor}
To prove Theorem \ref{sm}, it suffices to establish it for
\emph{excellent Morse} functions $f:X\rightarrow\R$\/, i.e. smooth
functions having only  non-degenerate critical points, each of which
corresponds to a different value of $f$\/.
\end{cor}

\begin{proof}
By a standard argument, given a non-degenerate quadratic form $Q$ on
$X$\/, the set of all continuous functions on $X$ equal to $Q$ off a
compact subset contains a $C^0$-dense subset consisting of excellent
Morse functions; our result follows by Proposition~\ref{st}.
\end{proof}

To prove Theorem \ref{sm} for excellent Morse functions, we will use
Morse homology.

\subsection*{Hypotheses}
We consider an excellent Morse function $f$ on $X$\/, quadratic at
infinity; for each pair of regular values $b<c$ of $f$\/, we denote
by $f_{b,c}$ the restriction of $f$ to $f^c\cap(-f)^{-b}=\{b\leq
f\leq c\}$\/.

\subsection*{Morse complexes}
Let
\[
C_k(f_{b,c}):=\{\xi_\ell^k:1\leq \ell\leq m_k\}
\]
denote the set of critical points of index $k$ of $f_{b,c}$\/,
ordered so that $f(\xi_{\ell}^k)< f(\xi_{m}^k)$ for $\ell<m$\/.
Given a generic gradient-like vector field $V$ for $f$ such that
$(f,V)$ is Morse-Smale\footnote{ Being Morse-Smale means that the
stable and unstable manifolds of all the critical points are
transversal.}, the \emph{Morse complex} of $(f_{b,c},V)$ over $R$
consists of the free $R$-modules
\[
M_k(f_{b,c},R):=\{\sum_{\ell} a_\ell \xi_\ell^k,\quad a_\ell\in R\}
\]
together with the boundary operator $\p: M_k(f_{b,c},R)\to
M_{k-1}(f_{b,c},R)$ given by
\[
\p \xi_\ell^k:=\sum_m \nu_{f,V}(\xi_\ell^k,\xi_{m}^{k-1})\xi_m^{k-1}
\]
where, with given orientations for the stable manifolds (hence
co-orientations for unstable manifolds), $\nu_{f,V}$ is the
intersection number of the stable manifold $W^s(\xi_l^k)$ of
$\xi_l^k$ and the unstable manifold $W^u(\xi_m^{k-1})$ of
$\xi_m^{k-1}$, i.e. the algebraic number of trajectories of $V$
connecting $\xi_\ell^k$ and $\xi_{m}^{k-1}$\/; note that
\begin{itemize}
  \item $\nu_{f,V}(\xi_\ell^k,\xi_{m}^{k-1})$ is the same for all $b,c$ with $f(\xi_\ell^k),f(\xi_{m}^{k-1})$ in
  $[b,c]$;
  \item $\nu_{f,V}(\xi_\ell^k,\xi_{m}^{k-1})\neq0$ implies $f(\xi_\ell^k)>f(\xi_{m}^{k-1})$\/: otherwise, the stable manifold of $\xi_{m}^{k-1}$ and the unstable manifold of $\xi_\ell^k$ for $V$\/, which cannot be transversal because of their dimensions, would intersect, contradicting the genericity of $V$\/.
  \item $\nu_{f,V}(\xi_l^k,\xi_m^k)=0$ for two distinct critical
  points of the same index.
\end{itemize}
This does define a complex, i.e. $\p\circ\p=0$\/: see for example
\cite{FL,ML}. The homology $HM_*(f_{b,c},R):= H_*(M_*(f_{b,c},R))$
is called the \emph{Morse homology}\footnote{Morse homology is
defined in general for any Morse function without being excellent.}
of $f_{b,c}$\/.

\begin{lem}[Barannikov,\cite{Br}]
\label{B} If $R$ is a field $\mathbb{F}$\/, then this boundary
operator $\p$ has a special kind of Jordan normal form as follows:
each $M_k(f_{b,c},\mathbb{F})$ has a basis
\begin{equation}
\label{eq31} \T^k_\ell:=\sum_{i\leq
\ell}\alpha_{\ell,i}\xi_i^k,\quad \alpha_{\ell,\ell}\neq 0
\end{equation}
such that either $\p \T_\ell^k=0$ or $\p \T_\ell^k=\T_m^{k-1}$ for
some $m$, in which case no $\ell'\neq\ell$ satisfies $\p
\T_{\ell'}^k=\T_m^{k-1}$\/. If $(\Theta_\ell^k)$ is another such
basis, then $\p \Xi_\ell^k=\Xi_m^{k-1} \;(\hbox{resp. } 0)$ is
equivalent to $\p\Theta_\ell^k=\Theta_m^{k-1}\;(\hbox{resp. } 0)$;
in other words, the matrix of $\p$ in all such bases is the same.
\end{lem}

\begin{proof} We prove existence by induction. Given nonegative integers $k,i$ with $i<m_k$\/, suppose that vectors $\Xi_q^p$ of the  form \eqref{eq31} have been obtained for all $(p,q)$ with either $p<k$, or $p=k$ and $q\leq i$\/, possessing the required property that either $\p \Xi_q^p=\Xi_{j_p(q)}^{p-1}$ (with $j_p(q)\neq j_p(q')$ for $q\neq q'$\/) or $\p \Xi_q^p=0$\/. If $\p\xi_{i+1}^k=0$ (e.g., when $k=0$\/), we take $\xi_{i+1}^k:=\Xi_{i+1}^k$ and continue the induction. Otherwise, $\p\xi_{i+1}^k=\sum \alpha_j \Xi_j^{k-1}$, $\alpha_j\in \mathbb{F}$. Moving all the terms $\Xi_{j_k(q)}^{k-1}=\p\Xi_q^k, q\leq i$ from the right-hand side to the left, we get
\[
\p\big(\xi_{i+1}^k-\sum_{q\leq i} \alpha_{j_k(q)}\Xi_q^k\big)=\sum_j
\beta_j \Xi_j^{k-1}\/.
\]
Let
\[
\Xi_{i+1}^k:=\xi_{i+1}^k-\displaystyle\sum_{q\leq i}
\alpha_{j_k(q)}\Xi_q^k\/.
\]
If $\beta_j=0$ for all $j$, then $\p\Xi_{i+1}^k=0$ and the induction
can go on. Otherwise,
\[
\p\Xi_{i+1}^k=\sum_{j\leq j_0}\beta_j
\Xi_j^{k-1}=:\tilde{\Xi}_{j_0}^{k-1} \hbox{ with } \beta_{j_0}\neq0;
\]
as $\p\tilde{\Xi}_{j_0}^{k-1}=\p\p\Xi_{i+1}^k=0$\/, we can replace
$\Xi_{j_0}^{k-1}$ by $\tilde{\Xi}_{j_0}^{k-1}$ and continue the
induction\footnote{Note that if $\mathbb{F}$ was not a field, this
would not provide a basis for noninvertible $\beta_{j_0}$.}.
\end{proof}

\begin{defn}
\label{def34} \rm Under the hypotheses and with the notation of the
Barannikov lemma, two critical points $\xi_\ell^k$ and $\xi_m^{k-1}$
of $f_{b,c}$ are \emph{coupled} if $\p \T_\ell^k=\T_m^{k-1}$. A
critical point is \emph{free} (over $\mathbb{F}$) when it is not
coupled with any other critical point.
\end{defn}

In other words, $\xi_\ell^k$ is free if and only if  $\T_\ell^k$ is
a cycle of $M_k(f_{b,c},\mathbb{F})$ but not a boundary, hence the
following result:

\begin{cor}
\label{cor35} For each integer $k$\/, the Betti number
$\dim_\mathbb{F}HM_k(f_{b,c},\mathbb{F})$ is the number of free
critical points of index $k$ of $f_{b,c}$ over $\mathbb{F}$\/.\qed
\end{cor}

\begin{thm}
\label{thm36}

\begin{enumerate}
  \item The Barannikov normal form of the Morse complex of $f_{b,c}$ over $\mathbb{F}$ is independent of the gradient-like vector field $V$.
  \item So is the Morse homology $HM_*(f_{b,c},R)$\/; it is isomorphic to $H_*(f^c,f^b;R)$\/.
  \item For $b'\leq b<c\leq c'$\/, the inclusion  $i:f^c\hookrightarrow f^{c'}$\/, restricted to the critical set $C_*(f_{b,c})$\/, induces a linear map $i_*:M_*(f_{b,c},R)\rightarrow M_*(f_{b',c'},R)$ such that $\p\circ i_*=i_*\circ \p$ and therefore a linear map $i_*:HM_*(f_{b,c},R)\rightarrow HM_*(f_{b',c'},R)$\/, which is the usual $i_*:H_*(f^c,f^b;R)\rightarrow H_*(f^{c'},f^{b'};R)$ modulo the isomorphism (ii).
\end{enumerate}
\end{thm}

\noindent{\emph{Idea of the proof}} \cite{FL}. $(1)$ Connecting two
generic gradient-like vector fields $V_0$, $V_1$ for $f$  by a
generic family, one can prove that each of the Morse complexes
defined by $V_0$ and $V_1$ is obtained from the other by a change of
variables whose matrix is upper-triangular with all diagonal entries
equal to $1$.

$(2)$ When there is no critical point of $f$ in $\{b\leq f\leq
c\}$\/, both $HM_*(f_{b,c},R)$ and $H_*(f^c,f^b;R)$ are trivial (the
flow of $V$ defines a retraction of $f^c$ onto $f^b$\/).

When there is only one critical point $\xi$ of $f$ in $\{b\leq f\leq
c\}$\/, of index $\lambda$\/,
\[
 HM_k(f_{b,c},R)\simeq H_k(f^c,f^b;R)\simeq\begin{cases}
      & R,\,\text{ if } k=\lambda, \\
      & 0\, \text{ otherwise}:
\end{cases}
\]
the class of $\xi$ obviously generates $HM_\lambda(f_{b,c},R)$\/,
whereas a generator of $H_\lambda(f^c,f^b;R)$ is the class of a cell
of dimension $\lambda$\/, namely the stable manifold of $\xi$ for
$V|_{\{b\leq f\leq c\}}$\/; the isomorphism associates the second
class to the first.

In the general case, one can consider a subdivision
$b=b_0<\cdots<b_N=c$ consisting of regular values of $f$ such that
each $f_{b_j,b_{j+1}}$ has precisely one critical point. One can
show that the boundary operator $\p$ of the relative singular
homology $\p: H_{k+1}(f^{b_{i+1}},f^{b_i})\to
H_k(f^{b_i},f^{b_{i-1}})$ can be interpreted as the intersection
number of the stable manifold of the critical point in $\{b_i\leq
f\leq b_{i+1}\}$ and the unstable manifold of that in $\{b_{i-1}\leq
f\leq b_i\}$, i.e., their algebraic number of connecting
trajectories.

$(3)$ The first claims are easy. The last one follows from what has
just been sketched. \qed

\begin{cor}
\label{cor37} If $f$ is an excellent Morse function quadratic at
infinity, then it has precisely one free critical point $\xi$ over
$\mathbb{F}$\/; its index $\lambda$ is that of the reference
quadratic form $Q$ and
\[
\underline{\gamma}(f,\mathbb{F})=f(\xi)\/.
\]
\end{cor}

\begin{proof}
Clearly, the dimension of
\[
HM_k(f,\mathbb{F})=HM_k(f_{-\infty,\infty},\mathbb{F})\simeq
H_k(f^\infty,f^{-\infty};\mathbb{F}) =
H_k(Q^\infty,Q^{-\infty};\mathbb{F})
\]
is $1$ if $k=\lambda$ and $0$ otherwise. The first two assertions
follow by Corollary~\ref{cor35}. To prove
$\underline{\gamma}(f,\mathbb{F)}=f(\xi)$\/, note that
$\underline{\gamma}(f)$ is the infimum of the regular values $c$ of
$f$ such that the class of $\xi$ in
$HM_{\lambda}(f_{-\infty,\infty},\mathbb{F})$ lies in the image of
$i_{c*}:HM_{\lambda}(f_{-\infty,c},\mathbb{F})\rightarrow
HM_{\lambda}(f_{-\infty,\infty},\mathbb{F})$\/; by
Theorem~\ref{thm36}~(iii), which means $c\geq f(\xi)$\/.
\end{proof}

%

\begin{prop}
\label{prop38} The excellent Morse function $-f_{b,c}=(-f)_{-c,-b}$
has the same free critical points over the field $\mathbb{F}$ as
$f_{b,c}$\/.
\end{prop}

\begin{proof}
Assuming $V$ fixed, this is essentially easy linear algebra:
\begin{itemize}
  \item One has $C_k(-f)=C_{n-k}(f)$ and the ordering of the corresponding critical values is reversed.  Thus, the lexicographically ordered basis of $M_*(-f)$ corresponding to $(\xi_\ell^k)_{1\leq \ell\leq m_k,0\leq k\leq n}$ is $(\xi_{m_{n-k}-\ell+1}^{n-k})_{1\leq \ell\leq m_{n-k},0\leq k\leq n}$\/.
  \item The vector field $-V$ has the same relations with $-f$ as $V$ has with $f$\/, hence $\nu_{-f,-V}(\xi_{m_{n-k}-\ell+1}^{n-k},\xi_{m_{n-(k-1)}-m+1}^{n-(k-1)})=\nu_{f,V}(\xi_{m_{n-(k-1)}-m+1}^{n-(k-1)},\xi_{m_{n-k}-\ell+1}^{n-k})$\/.
\end{itemize}
That is, the matrix of the boundary operator of $M_*(-f_{b,c})$ in
the basis $(\xi_{m_{n-k}-\ell+1}^{n-k})$ is the matrix $\tilde M$
obtained from the matrix $A$ of the boundary operator of
$M_*(f_{b,c})$ in the basis $(\xi_\ell^k)$ by symmetry with respect
to the second diagonal (i.e. by reversing the order of both the
lines and columns of the transpose of $A$\/).

Lemma~\ref{B} can be rephrased as follows: there exists a
block-diagonal matrix
\[
P=\hbox{diag}(P_0,\dots,P_n)
\]
where each $P_k\in \mathrm{GL}(m_k,\mathbb{F})$ is upper triangular,
such that
\begin{equation}
\label{eq32} P^{-1}AP=B
\end{equation}
is a Barannikov normal form, meaning the following: the entries of
the column of indices $_\ell^k$ are $0$ except possibly one, equal
to $1$\/, which must lie on the line of indices $_m^{k-1}$ for some
$m$ and be the only nonzero entry on this line. The normal form $B$
is the same for every choice of $P$ and $V$\/. Clearly, $\xi_\ell^k$
is a free critical point of $f_{b,c}$ if and only if both the line
and  column of indices $_\ell^k$ of $B$ are zero.

Equation \eqref{eq32} reads
\begin{equation}
\label{eq33} {\tilde P}{\tilde A}{\tilde P^{-1}}=\tilde B\/;
\end{equation}
Now, $\tilde P^{-1}$ and $\tilde P=(\tilde P^{-1})^{-1}$ are block
diagonal upper triangular matrices whose $k^{th}$ diagonal block
lies in $\mathrm{GL}(m_{n-k},\mathbb{F})$\/; therefore, by
\eqref{eq33}, as $\tilde B$ is a Barannikov normal form for the
ordering associated to $-f$\/, it is \emph{the} Barannikov normal
form of the boundary operator of $M_*(-f_{b,c})$\/, from which our
result follows at once.
\end{proof}

\begin{cor}
\label{cor39} For any excellent Morse function $f$ quadratic at
infinity, the sole free critical point of $-f$ over $\mathbb{F}$ is
the free critical point $\xi$ of $f$\/; hence
$\underline{\gamma}(f,\mathbb{F})=f(\xi)=-(-f)(\xi)=-\underline{\gamma}(-f,\mathbb{F})=\overline{\gamma}(f,\mathbb{F})$
by Corollary~\ref{cor37} and Lemma~\ref{lem29}, which proves
Theorem~\ref{sm}. \qed
\end{cor}

Before we give an example where $\m(f,\Z)>\s(f,\Z)$\/, here is a
situation where this cannot occur:

\begin{prop}\label{sd}
Assume that $M_*(f,\Z)$ can be put into Barannikov normal form by a
basis change \eqref{eq31} of the free $\Z$-module $M_*(f,\Z)$\/: \be
\label{bz}\T^k_\ell:=\sum_{i\leq
\ell}\alpha_{\ell,i}^k\xi_i^k,\quad\alpha_{\ell,i}^k\in\Z\/,\quad
\alpha_{\ell,\ell}^k=\pm1\/. \ee Then, $\m(f,\Z)=\s(f,\Z)=f(\xi)$\/,
where $\xi$ is the sole free critical point of $f$ over $\Z$\/.
\end{prop}

\begin{proof}
We are in the situation of the proof of Proposition~\ref{prop38}
with $P_k\in \mathrm{GL}(m_k,\Z)$\/, which implies that the
Barannikov normal form $B$ of the boundary operator is the same for
$\Z$ as for $\Q$\/; it does follow that there is a unique free
critical point $\xi$ of $f$ over $\Z$ (the same as over $\Q$\/) and
that it is the unique free critical point of $-f$ over $\Z$\/;
moreover, the proof of Corollary~\ref{cor37} shows that
$\m(f,\Z)=\s(f,\Z)=f(\xi)$\/. We conclude as in
Corollary~\ref{cor39}.
\end{proof}

Now that the coefficients are in $\Z$, the classical method of so
called {\it sliding handles} states that, under an additional
condition imposed on the index of the change of basis in (\ref{bz}),
namely $2\leq k\leq n-2$, the Barannikov normal form can be realized
by a gradient-like vector field for $f$.

More precisely, let $P: M_*(f)\to M_*(f)$ be a transformation matrix
where $P=\hbox{diag} (P_0,\dots ,P_n)$ with each $P_k\in GL(m_k,\Z)$
such that $P_k=id$ for $k=0,1$ or $n-1,n$, and $P_k$ is upper
triangular with $\pm 1$ in the diagonal entries for $2\leq k\leq
n-2$. Then one can construct a gradient-like vector field $V'$ such
that, if the matrix of the boundary operator for a given
gradient-like vector field $V$ is $A$, then the matrix for $V'$ is
given by $B=P^{-1}AP$.

Roughly speaking, one modifies $V$, each time for one $i\leq l$, by
sliding handle of the stable sphere\footnote{The stable and unstable
sphere is defined as : $S_L(\xi_l^k)=W^s(\xi_l^k)\cap L$ and
$S_R(\xi_i^k)=W^u(\xi_i^k)\cap L$ where $L=f^{-1}(c)$ for some $c\in
(f(\xi_i^k),f(\xi_l^k))$.}$S_L(\xi_l^k)$ of $\xi_l^k$ for $V$ such
that it sweeps across the unstable sphere $S_R(\xi_i^k)$ of
$\xi_i^k$ with indicated intersection number. In other words,
$S_L'(\xi_l^k)$ for the resulted $V'$ is the connected sum of
$S_L(\xi_l^k)$ and the boundary of a meridian disk of $S_R(\xi_i^k)$
described in section $4.4$ of \cite{FL}. One may refer to the Basis
Theorem (Theorem $7.6$) in \cite{ML} for a detailed construction of
$V'$.

\begin{rem}[on the ``proof'' of Corollary~\ref{cor39} in \cite{Ca}]
\label{rem42} \rm Capitanio uses the following\smallskip

\paragraph{\textsc{Criterion}}
\emph{A critical point $\xi$ of $f$ is free (over $\Q$\/) if and
only if , for any critical point $\eta$ incident to $\xi$,there is a
critical point $\xi'$, incident to $\eta$, such that}
\[
|f(\xi')-f(\eta)|<|f(\xi)-f(\eta)|\/.
\]
where fixing a gradient-like vector field $V$ generic for $f$, two
critical points are called incident if their algebraic number of
connecting trajectories is nonzero.

 Unfortunately, this is not true: one can construct a function
$f:\R^{2n}\to \R$\/, $n\geq 2$, quadratic at infinity with Morse
index $n$\/, having five critical points, two of index $n-1$ and
three of index $n$\/, whose gradient vector field $V$ defines the
Morse complex
\[
\p \xi_1^n=\xi_2^{n-1},\quad \p \xi_2^n=\xi_1^{n-1},\quad \p
\xi_3^n=0\/.
\]

This complex can be reformulated into  \beaa && \p
\xi_1^n=(\xi_2^{n-1}-\xi_1^{n-1})+ \xi_1^{n-1}\\
&& \p(\xi_2^n+\xi_1^n)=(\xi_2^{n-1}-\xi_1^{n-1})+2\xi_1^{n-1}\\
&& \p (\xi_3^n+\xi_2^n)=\xi_1^{n-1}\eeaa

Hence, for a change of basis \[\xi_2^{n-1}\mapsto
\xi_2^{n-1}-\xi_1^{n-1},\quad\xi_2^n\mapsto \xi_2^n+\xi_1^n,
\quad\xi_n^3\mapsto \xi_n^3+\xi_n^2\] one can construct a
gradient-like vector field $V'$ for $f$ by sliding handles, such
that
\[ \p \xi_1^n= \xi_2^{n-1}+\xi_1^{n-1},\quad \p
\xi_2^n=\xi_2^{n-1}+2 \xi_1^{n-1},\quad \p \xi_3^n=\xi_1^{n-1}\/.
\]

 Obviously, $\xi_3^n$ is the only free critical
point, but $\xi_2^n$ satisfies the criterion (with incidences under
$V'$). \hfill$\Box$
\end{rem}
\bigskip

\section{An example of Laudenbach}
\label{sec4}

\begin{prop}\label{p4}
There exists an excellent Morse function $f:\R^{2n}\rightarrow\R$ as
follows:
\begin{enumerate}
  \item it is quadratic at infinity and the reference quadratic form has index and coindex $n>1$\/;
  \item it has exactly five critical points:  three of index $n$\/, one of index $n-1$ and one of index $n+1$\/;
  \item its Morse complex over $\Z$ is given by
\bea \nonumber
&&\p\xi_1^{n-1}= 0\\
\label{eq44}&& \p \xi_1^n =\xi_1^{n-1}, \quad\p \xi_2^n= -2\xi_1^{n-1},\quad  \p \xi_3^n= -\xi_1^{n-1} \\
\nonumber &&\p \xi_1^{n+1}= \xi_2^n-2\xi_3^n\/, \eea
\end{enumerate}
hence, for any field $\mathbb{F}_2$ of characteristic $2$ and any
field $\mathbb{F}$ of characteristic $\neq2$\/,
\begin{equation}
\label{eq45}
\m(f,\Z)=\m(f,\mathbb{F}_2)=\s(f,\mathbb{F}_2)=f(\xi_3^n)\;>\;f(\xi_2^n)=\m(f,\mathbb{F})=\s(f,\mathbb{F})=\s(f,\Z).
\end{equation}
\end{prop}
\smallskip

\noindent{\emph{Proof that \eqref{eq44} implies \eqref{eq45}.} The
Morse complex of $f$ over $\mathbb{F}_2$ writes \beaa && \p
\xi_1^{n-1}=0
\\
&&\p \xi_1^n = \xi_1^{n-1}, \quad \p \xi_2^n= 0,\quad \p (\xi_3^n+\xi_1^n)=0\\
&& \p \xi_1^{n+1}= \xi_2^n\/, \eeaa implying that $\xi_3^n$ is the
only free critical point, hence, by Corollary~\ref{cor37},
\[
\m(f,\mathbb{F}_2)=\s(f,\mathbb{F}_2)=f(\xi_3^n)\/;
\]
as $\m(f,\Z)\geq\m(f,\mathbb{F}_2)$ by Proposition~\ref{prop211} and
$\m(f,\Z)\leq f(\xi_3^n)$\/, we do have
\[
\m(f,\Z)=f(\xi_3^n)\/.
\]

Similarly (keeping the numbering of the critical points defined by
$f$\/) the Morse complex of $-f$ over $\mathbb{F}$ has the
Barannikov normal form \beaa &&\p (-2\xi_1^{n+1})= 0
\\
&&\p \xi_3^n = -2\xi_1^{n+1},\quad\p (\xi_2^n+\frac{1}{2}\xi_3^n)=0,
\quad \p (-\xi_3^n-2\xi_2^n+\xi_1^n)=0
\\
&&\p \xi_1^{n-1}= -\xi_3^n-2\xi_2^n+\xi_1^n\/, \eeaa showing that
the free critical point is $\xi_2^n$\/; hence, by
Corollary~\ref{cor37} and Proposition~\ref{prop38},
\[
\s(f,\mathbb{F})=\m(f,\mathbb{F})=f(\xi_2^n)\/;
\]
finally, as we have $\s(f,\Z)\leq\s(f,\mathbb{F})$ by
Proposition~\ref{prop211}, and $\s(f,\Z)\geq f(\xi_1^n)$\/, we
should prove $\s(f,\Z)> f(\xi_1^n)$\/, which is obvious since
$\xi_1^n$ and $\xi_1^{n+1}$ are boundaries in $M_*(-f,\Z)$\/.
\medskip

\noindent{\emph{How to construct such a function $f$.} It is easy to
construct a function $f_0:\R^{2n}\to \R$ with properties $(1)$ and
$(2)$ required in the proposition and whose gradient vector field
$V_0$ provides a Morse complex given by \beaa &&\p \xi_1^{n-1}= 0,\\
&& \p\xi_1^n=\xi_1^{n-1},\quad \p\xi_2^n=0,\quad \p\xi_3^n=0\\&&
\p\xi_1^{n+1}=\xi_3^n.\eeaa

For a change of basis
\[\xi_2^n\mapsto \xi_2^n-\xi_1^n,\quad
\xi_3^n\mapsto\xi_3^n-2(\xi_2^n-\xi_1^n)\] one can construct a
gradient-like vector field $V'$ for $f_0$ by sliding handles, such
that \beaa && \p\xi_1^{n-1}=0\\&&\p \xi_1^n=\xi_1^{n-1},\quad
\p\xi_2^n=-\xi_1^{n-1},\quad \xi_3^n=-2\xi_1^{n-1}\\&&
\p\xi_1^{n+1}=-2\xi_2^n+\xi_3^n\eeaa

Since $(f_0,V')$ is Morse-Smale, the invariant manifolds of those
critical points of the same index are disjoint, hence one can modify
$f_0$ to $f$ such that
\begin{itemize}
\item $f$ has the same critical points of $f_0$;
\item the ordering of critical points for $f$ is $f(\xi_2^n)>f(\xi_3^n)>f(\xi_1^n)$,
\item $V'$ is a gradient-like vector field for $f$.
\end{itemize}
This can be realized by the preliminary rearrangement theorem
(Theorem $4.1$) in \cite{ML}.

In other words, we have made a change of critical points
$\xi_2^n\leftrightarrow \xi_3^n$, hence obtain the required Morse
complex in the proposition.

\bigskip

\paragraph{\textsc{Acknowledgements}} I wish to thank
F.Laundenbach for his example which asserts the difference between
minmax and maxmin in coefficients $\Z$. And I am also grateful for
the useful discussions with A. Chenciner and M.Chaperon.

\end{document}